\documentclass{article}
\usepackage{graphicx}
\usepackage{epsfig}
\usepackage{amssymb}
\usepackage{amsthm}
\usepackage{amsfonts}
\usepackage{amsmath}
\usepackage{multicol,multirow}
\usepackage[margin=1in]{geometry}
\usepackage{bbm} 
\usepackage{cases}	
\usepackage{float}
\newlength{\minuslength}
\usepackage{enumitem}
\usepackage{subcaption}
\usepackage{hhline}
\usepackage{relsize}

\usepackage{titling}
\newtheorem{theorem}{Theorem}

\newtheorem{proposition}[theorem]{Proposition}

\newtheorem*{affineiep}{Affine Inverse Eigenvalue Problem}

\newcommand{\Sym}{\mathbb{S}}
\newcommand{\trace}[1]{\mathrm{Tr}\left(#1\right)}

\newcommand{\R}{\mathbb{R}}
\newcommand{\Z}{\mathbb{Z}}
\newcommand{\eig}{\mathcal{E}}

\newcommand{\V}[1]{\mathcal{V}_{\mathbb{R}} \left(#1\right)}

\newcommand{\SH}{\textrm{SH}}

\title{A Note on Convex Relaxations for the Inverse Eigenvalue Problem}

\author{Utkan Onur Candogan\thanks{\noindent Department of Electrical Engineering,
California Institute of Technology, Pasadena, CA 91125, USA. \texttt{utkan@caltech.edu}} \hspace{0.25in} Yong Sheng Soh\thanks{Institute of High Performance Computing, 1 Fusionopolis Way, \# 16-16 Connexis, Singapore 138632. \texttt{soh\_yong\_sheng@ihpc.a-star.edu.sg}} \hspace{0.25in} Venkat Chandrasekaran\thanks{Department of Computing and Mathematical Sciences and Department of Electrical Engineering, California Institute of Technology, Pasadena, CA 91125, USA. \texttt{venkatc@caltech.edu} \newline The authors were supported in part by National Science Foundation grants CCF-1350590 and CCF-1637598, Air Force Office of Scientific Research grant FA9550-16-1-0210, and a Sloan Research Fellowship.}}

\begin{document}
\maketitle

\begin{abstract}
The affine inverse eigenvalue problem consists of identifying a real symmetric matrix with a prescribed set of eigenvalues in an affine space.  Due to its ubiquity in applications, various instances of the problem have been widely studied in the literature.  Previous algorithmic solutions were typically nonconvex heuristics and were often developed in a case-by-case manner for specific structured affine spaces.  In this short note we describe a general family of convex relaxations for the problem by reformulating it as a question of checking feasibility of a system of polynomial equations, and then leveraging tools from the optimization literature to obtain semidefinite programming relaxations.  Our system of polynomial equations may be viewed as a matricial analog of polynomial reformulations of $0/1$ combinatorial optimization problems, for which semidefinite relaxations have been extensively investigated.  We illustrate numerically the utility of our approach in stylized examples that are drawn from various applications.

\vspace{0.025in}


\vspace{0.025in}


\end{abstract}

\section{Introduction} \label{sec:intro}

The affine inverse eigenvalue problem (IEP) consists of identifying a real symmetric matrix with a prescribed set of eigenvalues in an affine space.  IEPs arise in a range of applications in engineering and physical sciences, such as natural frequency identification in vibrating systems, pole placement, factor analysis, reliability testing, estimation of the Earth's conductivity, graph partitioning and nuclear and molecular spectroscopy \cite{boley1987survey, chu1998inverse, chu2005inverse}.  Further, there are many situations in which a question of interest is to solve a discrete inverse Sturm-Liouville problem \cite{hald1977discrete}, which is a special case of an affine IEP.  Due to its ubiquity, IEPs have received much attention in the literature over the past several decades. On one end of the spectrum, there have been several efforts aimed at providing necessary and sufficient conditions for the existence of a solution of a given IEP \cite{ friedland1977inverse, hershkowitz1983existence, landau1994inverse}.  For example, Landau proved that there always exists a symmetric Toeplitz matrix with a desired set of eigenvalues \cite{landau1994inverse}; however, computing such matrices in a tractable manner remains a challenge. At the other end of the spectrum, several efforts have been aimed at developing efficient procedures for numerically finding solutions to particular types of the inverse eigenvalue problems \cite{ friedland1987formulation, wang1983numerical}, including some recent approaches based on convex optimization \cite{lin2010semi, zhao2014application}. Our work differs from these approaches in two prominent ways.  First, our framework is applicable to general affine IEPs, while some of the previous convex approaches are only useful for certain structured problem instances; see Section~\ref{SectionNumericalExperiments} for the broad range of examples to which we apply our methods.  Second, we describe a family of convex relaxations for IEPs rather than just a single convex program, and our work allows for a tradeoff between computational cost and solution quality.

We begin by first reformulating the affine IEP as a question of checking the existence of a real solution to a system of polynomial equations.  Formally, an instance of an affine IEP may be stated as follows:

\begin{affineiep}
Given $(i)$ a desired spectrum $\Lambda = \{(\lambda_i, m_i)\}_{i=1}^q \subset \R \times \Z_+$ of eigenvalue-multiplicity pairs with $\sum_i m_i = n$, and $(ii)$ an affine space $\mathcal{E} = \lbrace X \in \Sym^n~:~ \trace{C_k X} = b_k, ~k = 1,\dots,\ell \rbrace$, find an element of $\mathcal{E}$ with spectrum given by $\Lambda$ or certify that such a matrix does not exist.  Here $\Sym^n$ denotes the space of $n \times n$ real symmetric matrices.
\end{affineiep}

This problem may be reformulated as checking whether the following system has a real solution:
\begin{equation}\label{Zsystem}
\mathcal{S}_{\mathrm{iep}} := \begin{cases}
& f_1 := \, \sum_{i=1}^q Z_i - I  = 0 ,\\
& f_2^{(i)} \,:= \trace{Z_i} - m_i = 0 \text{ for } i = 1,\dots, q, \\
& f_3^{(i)} \,:= Z_i^2 - Z_i = 0  \text{ for } i = 1,\dots, q,\\
& f_4^{(k)} \,:= \sum_{i=1}^{q} \lambda_i \trace{ Z_i  C_k } - b_k  = 0 \text{ for } k = 1,\dots, \ell.
\end{cases}
\end{equation}
The variables in this system are the matrices $Z_1,\dots,Z_q \in \Sym^n$.  The matrix $I$ denotes the $n \times n$ identity.  We are concerned with whether the system of polynomials $\mathcal{S}_{\mathrm{iep}} = \{f_1, f^{(1)}_2, \dots, f^{(q)}_2, f^{(1)}_3, \dots, f^{(q)}_3, f^{(1)}_4, \dots, f^{(\ell)}_4\}$ has a common zero over the reals, or in other words checking whether the associated real variety $\mathcal{V}_\R(\mathcal{S}_{\mathrm{iep}})$ is empty.  It is clear that the system of equations \eqref{Zsystem} encodes the IEP.  The equations $f_1,\{f^{(i)}_2\}_{i=1}^q,\{f^{(i)}_3\}_{i=1}^q$ specify that the $Z_i$'s are projection matrices that partition the identity, and the equations $\{f^{(k)}_4\}_{k=1}^\ell$ requires that the matrix $\sum_i \lambda_i Z_i$ belongs to $\mathcal{E}$.  As such, an affine IEP is feasible if and only if $\mathcal{V}_\R(\mathcal{S}_{\mathrm{iep}}) \neq \emptyset$.

The advantage of this polynomial reformulation is that it allows us to leverage results from the optimization literature to systematically obtain convex relaxations for the affine IEP.  Specifically, Parrilo \cite{parrilo2000structured} and Lasserre \cite{lasserre2001global} developed hierarchies of semidefinite programming relaxations for polynomial optimization problems using results from real algebraic geometry.  These relaxations entail the solution of increasingly larger convex optimization problems that search over successively more complex collections of certificates that prove the infeasibility of the system defined by $\mathcal{S}_{\mathrm{iep}}$.  From a dual perspective, these relaxations may also be viewed as providing a sequence of convex outer approximations $\mathcal{R}_1(\Lambda,\mathcal{E}) \supseteq \mathcal{R}_2(\Lambda,\mathcal{E}) \supseteq \dots \supseteq \mathrm{conv}\left(\mathcal{V}_\R(\mathcal{S}_{\mathrm{iep}})\right)$, which leads to a natural heuristic for attempting to obtain solutions of the system $\mathcal{S}_{\mathrm{iep}}$.  We describe the mechanism to obtain these relaxations in Section~\ref{SectionRelaxations}.  As an illustration, searching over a simple class of infeasibility certificates gives the following convex outer approximation to $\mathcal{V}_\R(\mathcal{S}_{\mathrm{iep}})$:
\begin{equation}\label{EquationR1}
\begin{aligned}
\mathcal{R}_{1}(\Lambda,\eig) = \left\{(Z_1,\dots,Z_q)\in \otimes^q {\Sym^{n}} ~|~ \sum_{i=1}^q Z_i = I;  ~\trace{Z_i} =m_i,~ Z_i \succeq 0 ~ \forall i; ~ \trace{\sum_{i=1}^q \lambda_i Z_i C_k} = b_k ~\forall k \right\}.
\end{aligned}
\end{equation}
In Section~\ref{SectionR1} we relate the set $\mathcal{R}_{1}(\Lambda,\eig)$ to the \emph{Schur-Horn orbitope} \cite{sanyal2011orbitopes} associated to the spectrum $\Lambda$, which is the convex hull of all real symmetric matrices with spectrum $\Lambda$.  If $\mathcal{R}_{1}(\Lambda,\eig) = \emptyset$, then it is clear that $\mathcal{V}_\R(\mathcal{S}_{\mathrm{iep}}) = \emptyset$; otherwise, $\mathcal{V}_\R(\mathcal{S}_{\mathrm{iep}})$ may or may not be empty, and one can either attempt to find an element of $\mathcal{V}_\R(\mathcal{S}_{\mathrm{iep}})$ or search over a larger family of infeasibility certificates (see Section~\ref{realag}).  In Section~\ref{SectionR2} we describe a convex outer approximation $\mathcal{R}_2(\Lambda,\mathcal{E})$ to $\mathrm{conv}\left(\mathcal{V}_\R(\mathcal{S}_{\mathrm{iep}})\right)$ that is in general tighter than $\mathcal{R}_1(\Lambda,\mathcal{E})$.

The description of $\mathcal{R}_1(\Lambda,\mathcal{E})$ in \eqref{EquationR1} consists of $q$ semidefinite constraints on matrix variables of size $n \times n$.  The description of $\mathcal{R}_2(\Lambda,\mathcal{E})$ in Section~\ref{SectionR2} involves a semidefinite constraint on a matrix variable of size ${n+1 \choose 2} q \times {n+1 \choose 2} q$.  Tighter relaxations to $\mathrm{conv}\left(\mathcal{V}_\R(\mathcal{S}_{\mathrm{iep}})\right)$ than $\mathcal{R}_1(\Lambda,\mathcal{E})$ and $\mathcal{R}_2(\Lambda,\mathcal{E})$ require even larger semidefinite descriptions, and they become prohibitively expensive to solve for large $n$.  Consequently, although we describe the general mechanism by which semidefinite relaxations of increasing size may be generated, we restrict our attention in numerical experiments to the performance of the relaxations $\mathcal{R}_1(\Lambda,\mathcal{E})$ and $\mathcal{R}_2(\Lambda,\mathcal{E})$.  As the affine IEP includes (co-)NP-hard problems as special cases, these two relaxations generally do not solve every instance of an affine IEP (as expected); nonetheless, we demonstrate their effectiveness in Section~\ref{SectionNumericalExperiments} on stylized problems such as certifying non-existence of planted subgraphs, solving discrete Sturm-Liouville equations, and computing Toeplitz matrices with a desired spectrum.

\paragraph{Connection to combinatorial optimization} A number of combinatorial problems such as computing the stability number of a graph or the knapsack problem may be formulated as checking feasibility of a system of equations in a collection of variables that take on values of $0/1$.  As many of these problems are NP-hard, a prominent approach to developing tractable approximations is to first specify the problems via polynomial equations and to then employ the methods referenced above to obtain semidefinite relaxations \cite{blekherman2012semidefinite}.  The polynomial reformulations consist of a system of equations defined by affine polynomials and quadratic equations of the form $x_i^2 - x_i = 0$ for each of the variables $x_i$ to enforce the Boolean constraints.  Our system \eqref{Zsystem} for the affine IEP may be viewed as a matricial analog of those arising in the literature on combinatorial problems, as the idempotence constraints $Z_i^2 - Z_i = 0$ represent a generalization of the scalar Boolean constraints $x_i^2 - x_i = 0$.  The present note describes promising experimental results of the performance of semidefinite relaxations for the affine IEP.  As with the significant prior body of work on combinatorial optimization, it is of interest to investigate structural properties of our relaxations for specific affine spaces $\eig$ and spectra $\Lambda$.  We outline future directions along these lines in Section~\ref{sec:conc}.

\section{Semidefinite Relaxations for Affine IEPs} \label{SectionRelaxations}

\subsection{From Polynomial Formulations to Semidefinite Relaxations} \label{realag}
We summarize here the basic aspects of obtaining semidefinite relaxations for certifying infeasibility of a polynomial systems over the reals; we refer the reader to the survey \cite{blekherman2012semidefinite} for further details.  Let $\mathbb{R}[x]$ denote the ring of polynomials with real coefficients in indeterminates $x = (x_1,\dots, x_n)$.  A \emph{polynomial ideal} $\mathcal{I}$ is a subset of $\mathbb{R}[x]$ that satisfies the following properties: $(i)$ $0 \in \mathcal{I}$, $(ii)$ $f_1,f_2 \in \mathcal{I} \Rightarrow f_1 + f_2 \in \mathcal{I}$, and $(iii)$ $f \in \mathcal{I}, h \in \mathbb{R}[x] \Rightarrow hf \in \mathcal{I}$.  The ideal \emph{generated} by a collection of polynomials $f_1, \dots, f_t \in \mathbb{R}[x]$ is the set $\langle f_1,\dots,f_t \rangle = \lbrace \sum_{i=1}^t f_i h_i : h_i \in \mathbb{R}[x]\rbrace$ -- here, $f_1,\dots,f_t$ and $h_1,\dots,h_t$ are referred to as \emph{generators} and \emph{coefficients}, respectively.  The \emph{real variety} corresponding to polynomials $g_1, \dots, g_r \in \mathbb{R}[x]$ is denoted $\V{g_1,\dots,g_r} = \lbrace x \in \mathbb{R}^n :  g_i(x) = 0, ~ i = 1,\dots,r \rbrace$.  Finally, the set of polynomials that can be expressed as a \emph{sum of squares} of polynomials is denoted $\Sigma := \lbrace p \in \mathbb{R}[x] : p = \sum_i p_i^2, ~ p_i \in \mathbb{R}[x] \rbrace$.  We state next the \emph{real Nullstellensatz} due to Krivine for certifying infeasibility of a system of a polynomial equations over $\R$:
\begin{theorem}[Real Nullstellensatz] 
Given any collection of polynomials $f_1,\dots,f_t \in \mathbb{R}[x]$, we have that:
\begin{align*}
-1 \in \Sigma +  \langle f_1,\dots,f_t \rangle  \iff \V{f_1,\dots,f_t} = \emptyset.
\end{align*}
\end{theorem}

Here $-1 \in \mathbb{R}[x]$ refers to the constant polynomial.  The implication that $-1 \in \Sigma +  \langle f_1,\dots,f_t \rangle  \Rightarrow \V{f_1,\dots,f_t} = \emptyset$ is straightforward.  The reverse direction may be proved by appealing to Tarski's transfer principle.  In general, the best-known bounds on the size of infeasibility certificates -- i.e., the degrees of the polynomials in $\Sigma, \langle f_1,\dots,f_t \rangle$ that sum to $-1$ -- are at least triply exponential.  This is to be expected as many co-NP-hard problems can be reformulated as certifying infeasibility of a system of polynomial equations.

Obtaining \emph{tractable relaxations} based on the real Nullstellensatz relies on three key observations.  First, one fixes a subset $\tilde{\mathcal{I}} \subset \langle f_1,\dots,f_t \rangle$ by considering polynomials $\sum_i h_i f_i$ in which the coefficients $h_i \in \mathbb{R}[x]$ have bounded degree (sets of the form $\tilde{\mathcal{I}}$ are sometimes called \emph{truncated ideals}, although they are not formally ideals).  In searching for infeasibility certificates of the form $-1 = p + q, ~ p \in \Sigma, q \in \tilde{I}$, one can check that without loss of generality the search for $p$ can also be restricted to sum-of-squares polynomials of bounded degree; formally, if every element of $\tilde{I}$ has degree at most $2d$, one can restrict the search to elements of $\Sigma$ with degree at most $2d$.  Second, a decomposition $-1 = p + \sum_i h_i f_i$ where the $p$ and the $h_i's$ all have bounded degree is a linear constraint in the coefficients of $p$ and the $h_i$'s.  Finally, checking that a polynomial $p \in \mathbb{R}[x]$ in $n$ variables of degree at most $2d$ is an element of $\Sigma$ can be formulated as a semidefinite feasibility problem; letting $m_{n,d}(x)$ denote the vector of all ${n+d \choose d}$ monomials in $n$ variables of degree at most $d$, we have that:
\begin{equation*}
p \in \Sigma ~~~ \Leftrightarrow ~~~ \exists P \in \mathbb{S}^{{n+d \choose d}}, ~ P \succeq 0, ~ p(x) = m_{n,d}(x)' ~ P ~ m_{n,d}(x).
\end{equation*}
The relation $p(x) = m_{n,d}(x)' ~ P ~ m_{n,d}(x)$ is equivalent to a set of linear equations relating the entries of $P$ to the coefficients of $p$.  Thus, the search over a restricted family of infeasibility certificates via bounding the degree of the coefficients of the elements of $\langle f_1,\dots,f_t \rangle$ is a semidefinite feasibility problem.

By considering a sequence of degree-bounded subsets $\mathcal{I}' \subset \mathcal{I}'' \subset \dots \subset \langle f_1,\dots,f_t \rangle$, one can search for more complex infeasibility certificates at the expense of solving increasingly larger semidefinite programs.  Associated to this sequence of semidefinite programs is a sequence of \emph{dual optimization problems} that provide successively tighter convex outer approximations to $\mathcal{V}_\R(f_1,\dots,f_t)$ (assuming strong duality holds), i.e., $\mathcal{R}' \supseteq \mathcal{R}'' \supseteq \dots \supseteq \mathrm{conv}\left(\mathcal{V}_\R(f_1,\dots,f_t)\right)$.  This dual perspective is especially interesting for attempting to identify elements of $\mathcal{V}_\R(f_1,\dots,f_t)$.  Concretely, fix a subset $\mathcal{I}' \subset \langle f_1,\dots,f_t \rangle$, and suppose that the search for an infeasibility certificate of the form $-1 \in \Sigma + \mathcal{I}'$ is unsuccessful.  Then $\mathcal{V}_\R(f_1,\dots,f_t)$ may or may not be empty.  At this stage, one can attempt to find an element of $\mathcal{V}_\R(f_1,\dots,f_t)$ by optimizing a random linear functional over the set $\mathcal{R}'$ (obtained by considering the dual problem associated to the system $-1 \in \Sigma + \mathcal{I}'$), and checking whether the resulting optimal solution lies in $\mathcal{V}_\R(f_1,\dots,f_t)$; this heuristic is natural as $\mathcal{R}' \supseteq \mathrm{conv}\left(\mathcal{V}_\R(f_1,\dots,f_t)\right)$, and if these sets were equal then the heuristic would generically succeed at identifying an element of $\mathcal{V}_\R(f_1,\dots,f_t)$.  If this approach to finding a solution is also unsuccessful, one can consider a larger subset $\mathcal{I}'' \subset \langle f_1,\dots,f_t \rangle$ and an associated tighter approximation $\mathcal{R}'' \supseteq \mathrm{conv}\left(\mathcal{V}_\R(f_1,\dots,f_t)\right)$ (here $\mathcal{I}' \subset \mathcal{I}''$ and $\mathcal{R}' \supseteq \mathcal{R}''$), and repeat the above procedure at a greater computational expense.

In Sections~\ref{SectionR1} and \ref{SectionR2}, we employ the methodology described above to give concrete descriptions of two convex outer approximations of the variety specified by the system $\mathcal{S}_{\mathrm{iep}}$ associated to the affine IEP.

\subsection{A First Semidefinite Relaxation}
\label{SectionR1}

As our first example, we consider the following truncated ideal associated to the system $\mathcal{S}_{\mathrm{iep}}$:
\begin{equation}\label{EquationI1}
\begin{aligned}
\mathcal{I}_1 =  \Big\{ \trace{h_1 f_1  } + \sum_{i=1}^q \left[h_2^{(i)} f_2^{(i)} + \trace{h_3^{(i)} f_3^{(i)}}\right]  & + \sum_{k=1}^\ell h_4^{(k)} f_4^{(k)} : h_2^{(i)}, h_4^{(k)} \in \R, ~ h_1, h_3^{(i)} \in \Sym^n ~ \forall i, k \\ & h_1, h_2^{(i)}, h_3^{(i)}, h_4^{(k)} \text{ do not depend on } Z_i \Big\}.
\end{aligned}
\end{equation}
In words, the truncated ideal $\mathcal{I}_1 \subset \langle \mathcal{S}_{\mathrm{iep}} \rangle$ is obtained by restricting the coefficients to be constant polynomials (i.e., degree-zero polynomials).  As a result, the elements of $\mathcal{I}_1$ consist of polynomials with degree at most two.  Consequently, in searching for infeasibility certificates of the form $-1 \in \mathcal{I}_1 + \Sigma$ one need only consider quadratic polynomials in $\Sigma$, which in turn leads to checking feasibility of the following semidefinite program:
\begin{equation}
\begin{aligned}\label{AltSystemEqs}
-\trace{A} - \sum_{i=1}^q m_i d_i - \sum_{k=1}^\ell b_k \xi_k &= 1; ~~~
A + d_i I  + \lambda_i \sum_{k=1}^\ell \xi_k  C_k - B_{ii} &= 0, ~ B_{ii} \succeq 0, ~ i=1,\dots,q
\end{aligned}
\end{equation}
in variables $A \in \mathbb{S}^n$, $d_i \in \R$ and $B_{ii} \in \mathbb{S}^n$ for $i=1,\dots,q$,  and $\xi_k\in \R$ for $k=1,\dots,\ell$.  The elements of the truncated ideal $\mathcal{I}_1$ can be associated to the above problem via the relations $h_1 = -A,~ h_2^{(i)}= -d_i ,~ h_3^{(i)}= - B_{ii} , ~h_4^{(k)}= -\xi_k$, and then observing that the constraints in \eqref{AltSystemEqs} are equivalent to checking that the polynomial $\trace{h_1 f_1  } + \sum_{i=1}^q h_2^{(i)} f_2^{(i)}+ \sum_{i=1}^q \trace{h_3^{(i)} f_3^{(i)}  }  + \sum_{k=1}^\ell f_4^{(k)} h_4^{(k)}$ in variables $(Z_1,\dots,Z_q)$ can be decomposed as $-1 - \Sigma$.  Next we relate $\mathcal{R}_1(\Lambda,\eig)$ and the system $-1 \in \mathcal{I}_1 + \Sigma$ via strong duality:
\begin{proposition} \label{ThmSystAlt1}
Consider an affine IEP specified by a spectrum $\Lambda$ and an affine space $\eig \subset \mathbb{S}^n$.  Let $\mathcal{I}_1$ be defined as in \eqref{EquationI1} and $\mathcal{R}_1(\Lambda,\eig)$ as in \eqref{EquationR1}.  Then exactly one of the following two statements is true:
\begin{equation*}
(1) ~ \mathcal{R}_1(\Lambda,\eig) \text{ is nonempty}, \hspace{1in} (2) ~ -1 \in \mathcal{I}_1 + \Sigma.
\end{equation*}
\end{proposition}

\begin{proof}
The feasibility of the system \eqref{AltSystemEqs} is equivalent to the condition $-1 \in \mathcal{I}_1 + \Sigma$.  One can check that the system \eqref{AltSystemEqs} and the constraints describing $\mathcal{R}_1(\Lambda,\eig)$ are \emph{strong alternatives} of each other, which follows from an application of conic duality -- strong duality follows from Slater's condition being satisfied.
\end{proof}

As a consequence of this result, it follows that $\mathcal{R}_1(\Lambda,\eig)$ is a convex outer approximation of the variety $\mathcal{V}_\R(\mathcal{S}_{\mathrm{iep}})$.  A more direct way to see this is to consider any element $(Z_1,\dots,Z_q) \in \mathcal{V}_\R(\mathcal{S}_{\mathrm{iep}})$ and to note that the idempotence constraints $Z_i^2 - Z_i = 0$ in $\mathcal{S}_{\mathrm{iep}}$ imply the semidefinite constraints $Z_i \succeq 0$ in $\mathcal{R}_1(\Lambda,\eig)$.

The set $\mathcal{R}_1(\Lambda,\eig)$ is closely related to the Schur-Horn orbitope associated to the spectrum $\Lambda$ \cite{sanyal2011orbitopes}:
\begin{equation}\label{Schur-Horn}
\begin{aligned}
\SH(\Lambda) &= \mathrm{conv}\{M \in \Sym^n ~|~ \lambda(M) = \Lambda\} \\ &= \left\{X \in \Sym^n ~|~ \exists (Z_1,\dots,Z_q) \in \otimes^q {\Sym^{n}} \text{ s.t. } \sum_{i=1}^q Z_i = I;  ~\trace{Z_i} =m_i,~ Z_i \succeq 0 ~ \forall i; ~ X = \sum_{i=1}^q \lambda_i Z_i \right\}. 
\end{aligned}
\end{equation}
The second equality follows from the characterization in \cite{ding2009low}.  The Schur-Horn orbitope was so-named by the authors of \cite{sanyal2011orbitopes} due to its connection with the Schur-Horn theorem.  A subset of the authors of the present note employed the Schur-Horn orbitope in developing efficient convex relaxations for NP-hard combinatorial optimization problems such as finding planted subgraphs \cite{candogan2018finding} and computing edit distances between pairs of graphs \cite{candogan2019convex}.  In the context of the present note, the Schur-Horn orbitope provides a precise characterization of the conditions under which $-1 \in \mathcal{I}_1 + \Sigma$ is successful.  Specifically, from Proposition~\ref{ThmSystAlt1} and \eqref{Schur-Horn} we have that:
\begin{equation}
-1 \in \mathcal{I}_1 + \Sigma \quad \Leftrightarrow \quad \mathcal{R}_1(\Lambda,\eig) = \emptyset \quad \Leftrightarrow \quad \SH(\Lambda) \cap \eig = \emptyset. \label{schurhornsuccess}
\end{equation}
Hence, if $-1 \notin \mathcal{I}_1 + \Sigma$, we have that $\mathcal{R}_1(\Lambda,\eig) \neq \emptyset$.  In particular, the variety $\mathcal{V}_\R(\mathcal{S}_{\mathrm{iep}})$ may or may not be empty.  At this stage, as discussed in Section~\ref{realag} one can maximize a random linear functional over the set $\mathcal{R}_1(\Lambda,\eig)$; the resulting optimal solution $(\hat{Z}_1,\dots,\hat{Z}_q)$ is generically an extreme point of $\mathcal{R}_1(\Lambda,\eig)$, and one can check if $(\hat{Z}_1,\dots,\hat{Z}_q)$ satisfies the equations in the system $\mathcal{S}_{\mathrm{iep}}$.  If this attempt at finding a feasible point in $\mathcal{V}_\R(\mathcal{S}_{\mathrm{iep}})$ is unsuccessful, one can repeat the preceding steps at attempting to certify infeasibility or to find a feasible point in $\mathcal{V}_\R(\mathcal{S}_{\mathrm{iep}})$ via a larger semidefinite program, which we describe in the next subsection.

We present here a result on guaranteed recovery of a solution to an affine IEP when the affine space is a random subspace:
\begin{proposition}
Let $X^\star \in \mathbb{S}^n$ have a spectrum $\Lambda$ with $n$ distinct eigenvalues, and suppose $\mathcal{E} = \{X ~|~ \mathrm{Tr}(C_k X) = \mathrm{Tr}(C_k X^\star), ~ k=1,\dots,\ell\}$ is an affine space with the $C_k \in \mathbb{S}^n$ being random matrices with i.i.d. standard Gaussian entries.  If $\ell > {n+1 \choose 2} - H_n$ where $H_n = \sum_{j=1}^n \tfrac{1}{n}$ is the $n$'th harmonic number, then with high probability the unique element of $\mathcal{R}_1(\Lambda,\eig)$ is the set of $n$ projection maps onto the eigenspaces of $X^\star$.
\end{proposition}
\begin{proof}
As Gaussian random matrices constitute an orthogonally invariant ensemble, we can assume without loss of generality that $X^\star$ is a diagonal matrix with the eigenvalues in descending order on the diagonal.  From \eqref{EquationR1}, \eqref{Schur-Horn}, and \eqref{schurhornsuccess}, we need to ensure that $\SH(\Lambda) \cap \eig = \{X^\star\}$.  From the results in \cite{almt2014edge,crpw2012atomic}, we have that if $\eta$ is the expected value of the square of the Euclidean distance of a Gaussian random matrix to the normal cone at $X^\star$ with respect to $\SH(\Lambda)$, then $\SH(\Lambda) \cap \eig = \{X^\star\}$ with high probability provided $\ell > \eta$.  From \cite{candogan2018finding} we have that the normal cone at $X^\star$ with respect to $\SH(\Lambda)$ is the set of diagonal matrices with the diagonal entries sorted in descending order.  From \cite{almt2014edge} we have that the expected squared Euclidean distance of a random Gaussian matrix to such a cone of sorted entries is equal to ${n+1 \choose 2} - H_n$.
\end{proof}

\subsection{A Tighter Semidefinite Relaxation} \label{SectionR2}


Next we consider a larger truncated ideal $\mathcal{I}_2 \subset \langle \mathcal{S}_{\mathrm{iep}} \rangle$ with larger degree coefficients than in $\mathcal{I}_1$:
\begin{equation} \label{eq:I2}
\begin{aligned}
\mathcal{I}_2 =  \Big\{ &\trace{h_1\left( Z_1,\dots,Z_q \right) f_1  } + \left[\sum_{i=1}^q h_2^{(i)}\left( Z_1,\dots,Z_q \right) f_2^{(i)}+  \trace{h_3^{(i)} f_3^{(i)}} \right]  + \sum_{k=1}^l f_4^{(k)} h_4^{(k)}\left( Z_1,\dots,Z_q \right) : \\
& h_1\left( Z_1,\dots,Z_q \right), h_3^{(i)} \in \Sym^n, ~ h_2^{(i)} \left( Z_1,\dots,Z_q \right), h_4^{(k)} \left( Z_1,\dots,Z_q \right) \in \R, ~ \forall i, k \\ & h_1\left( Z_1,\dots,Z_q \right), h_2^{(i)} \left( Z_1,\dots,Z_q \right), h_4^{(k)} \left( Z_1,\dots,Z_q \right) \text{ affine in } Z_i, ~ h_3^{(i)} \text{ does not depend on } Z_i \Big\}.
\end{aligned}
\end{equation}
The coefficients $h_3^{(i)}$ are constrained in the same way as in
$\mathcal{I}_1$ but the other coefficients $h_1, h_2^{(i)}, h_4^{(k)}$ are allowed to be affine polynomials (in the case of $h_1$, more precisely a matrix of affine polynomials).  The resulting collection $\mathcal{I}_2$ consists of polynomials of degree at most two, and therefore we can restrict our attention to elements of $\Sigma$ of degree at most two in searching for infeasibility certificates of the form $-1 \in \mathcal{I}_2 + \Sigma$.  However, $\mathcal{I}_2$ is in general larger than $\mathcal{I}_1$ so that $\mathcal{I}_2 + \Sigma$ offers a richer family of infeasibility certificates than $\mathcal{I}_1 + \Sigma$.  The convex relaxation $\mathcal{R}_2(\Lambda,\eig)$ obtained as an alternative to the system $-1 \in \mathcal{I}_2 + \Sigma$ in turn provides a tighter approximation in general than $\mathcal{R}_1(\Lambda,\eig)$ to the convex hull $\mathrm{conv}(\mathcal{V}_\R(\mathcal{S}_{\mathrm{iep}}))$.  We require some notation to give a precise description of $\mathcal{R}_2(\Lambda,\eig)$.  Let $\delta_{k,l}$ denote the usual delta function which equals one if the arguments are equal and zero otherwise. Additionally, for $s,t = 1,\dots,n$ let
\begin{align*}
f_{s,t} = \begin{cases} e_s e_t^T, \quad \quad \quad \quad ~~ \text{ if }  s=t,\\ \frac{1}{2}( e_se_t^T + e_t e_s^T), \text{ otherwise.}  \end{cases}
\end{align*}
Here, $e_s,e_t \in \mathbb{R}^n$ are the $s$'th and $t$'th standard basis vectors in $\R^n$.  The set $\mathcal{R}_2(\Lambda,\eig)$ is then specified as:
\begin{equation} \label{EqSecondDegree}
\begin{aligned}
\mathcal{R}_2(\Lambda,\eig) =\Big \lbrace   & (Z_1,\dots,Z_q)\in \otimes^q {\Sym^{n}} ~|~  \exists \, \mathcal{W}_{i,j} : \mathbb{S}^n \rightarrow \mathbb{S}^n, i,j = 1,\dots,q, ~ \exists \, \mathfrak{W}:  \otimes^q {\Sym^{n}} \rightarrow \otimes^q {\Sym^{n}},\\
& \mathfrak{W} \succeq 0, ~ [\mathfrak{W}(X_1,\dots,X_q)]_{i} = \sum_{j=1}^q \mathcal{W}_{i,j}(X_j) ~  i = 1,\dots,q, ~Z_i \succeq 0 ~  i = 1,\dots,q\\
& \sum_{i=1}^{q} Z_i = I,~ \trace{Z_i} = m_i, ~ i = 1,\dots,q,~ \sum_{i=1}^q \lambda_i \trace{ Z_i  C_k } = b_k, k = 1,\dots,l,\\
& \sum_{j=1}^q \mathcal{W}_{i,j}(f_{s,t}) = \delta_{s,t} Z_i, ~ s,t = 1,\dots,n, ~ i = 1,\dots,q, \\
& \sum_{s=1}^n \mathcal{W}_{i,j}(f_{s,s}) =  m_j Z_i, ~ i,j  = 1,\dots,q, \\
& \sum_{r=1}^n \trace{  f_{s,r} \mathcal{W}_{i,i}(f_{t,r})   } = (Z_i)_{s,t}, ~ i = 1,\dots,q,~ s,t = 1,\dots,n,\\
& \sum_{j=1}^q \lambda_j \mathcal{W}_{i,j}(C_k) = b_k Z_i, ~ i = 1,\dots,q, ~ k = 1,\dots,\ell~
      \Big\rbrace.
\end{aligned}
\end{equation}
Our next result records the fact that $\mathcal{R}_2(\Lambda,\eig)$ does constitute a strong alternative for $-1 \in \mathcal{I}_2 +\Sigma$.
\begin{proposition} \label{ThmSystAlt2}
Consider an affine IEP specified by a spectrum $\Lambda$ and an affine space $\eig \subset \mathbb{S}^n$.  Let $\mathcal{I}_2$ be defined as in \eqref{eq:I2} and $\mathcal{R}_2(\Lambda,\eig)$ as in \eqref{EqSecondDegree}.  Then exactly one of the following two statements is true:
\begin{equation*}
(1) ~ \mathcal{R}_2(\Lambda,\eig) \text{ is nonempty}, \hspace{1in} (2) ~ -1 \in \mathcal{I}_2 +\Sigma.
\end{equation*}
\end{proposition}
\begin{proof}
The proof is identical to that of Proposition~\ref{ThmSystAlt1}, and it follows from an application of conic duality.
\end{proof}
It is clear that $\mathcal{R}_1(\Lambda,\eig) \supseteq \mathcal{R}_2(\Lambda,\eig)$ as the constraints defining $\mathcal{R}_2(\Lambda,\eig)$ are a superset of those defining $\mathcal{R}_1(\Lambda,\eig)$. Further, for any $(Z_1,\dots,Z_q) \in \mathcal{V}_\R(\mathcal{S}_{\mathrm{iep}})$, one can check that the constraints defining $\mathcal{R}_2(\Lambda,\eig)$ are satisfied by setting the linear operators $\mathcal{W}_{i,j}(X) = \trace{ Z_j X} Z_i~ \forall X \in \Sym^n$.  Thus, there are additional \emph{quadratic relations} among the $Z_i$'s that are satisfied by the elements of $\mathcal{V}_\R(\mathcal{S}_{\mathrm{iep}})$ and are implied by $\mathcal{R}_2(\Lambda,\eig)$, but are not captured by the set $\mathcal{R}_1(\Lambda,\eig)$.  This is the source of the improvement of the relaxation $\mathcal{R}_2(\Lambda,\eig)$ compared to $\mathcal{R}_1(\Lambda,\eig)$, although the improvement comes at the expense of solving a substantially larger semidefinite program.  In particular, $\mathcal{R}_1(\Lambda,\eig)$ entails checking $q$ semidefinite constraints on $n \times n$ real symmetric matrices, while the description of $\mathcal{R}_2(\Lambda,\eig)$ involves a semidefinite constraint on the operator $\mathfrak{W}:  \otimes^q {\Sym^{n}} \rightarrow \otimes^q {\Sym^{n}}$, which is equivalent to stipulating that a ${n+1 \choose 2} q \times {n+1 \choose 2} q$ real symmetric matrix is positive semidefinite.  Thus, optimizing over $\mathcal{R}_2(\Lambda,\eig)$ is much more computationally expensive than $\mathcal{R}_1(\Lambda,\eig)$.

\section{Numerical Illustrations} \label{SectionNumericalExperiments}

Here we present experiments illustrating the performance of the relaxations $\mathcal{R}_1(\Lambda,\eig), \mathcal{R}_2(\Lambda,\eig)$ on random problem instances and stylized instances arising in applications.  Our results are obtained using the CVX parser \cite{grant2008graph} and the SDPT3 solver \cite{toh1999sdpt3}.  Before describing these, we present an approach to strengthen the relaxation $\mathcal{R}_2(\Lambda,\eig)$ by adding valid constraints without increasing the size of the semidefinite inequality.


\subsection{Strengthening the Relaxations}
A prominent approach in the optimization literature for obtaining improved bounds on hard nonconvex problems is to add redundant constraints.  The procedure presented in Section~\ref{realag} of considering a sequence of truncated ideals $\mathcal{I}_1 \subseteq \mathcal{I}_2 \subseteq \dots \subseteq \langle \mathcal{S}_{\mathrm{iep}} \rangle$ is a systematic method to add valid constraints; in particular, the elements of $\mathcal{I}_1$ and $\mathcal{I}_2$ represent polynomials that vanish at all the points in $\mathcal{V}_\R(\mathcal{S}_{\mathrm{iep}})$.  As $\mathcal{I}_1 \subseteq \mathcal{I}_2$, the relaxation $\mathcal{R}_2(\Lambda,\eig)$ offers (in general) a tighter convex outer approximation of $\mathcal{V}_\R(\mathcal{S}_{\mathrm{iep}})$ than $\mathcal{R}_1(\Lambda,\eig)$ as $\mathcal{R}_2(\Lambda,\eig)$ is derived from the incorporation of a larger collection of redundant constraints.

Here we present a simple alternative approach to adding redundant constraints for the affine IEP by augmenting the original system $\mathcal{S}_{\mathrm{iep}}$ with additional polynomials that vanish on $\mathcal{V}_\R(\mathcal{S}_{\mathrm{iep}})$, and which are not contained in the truncated ideals $\mathcal{I}_1, \mathcal{I}_2$.  Specifically, we consider the following modified system of equations:
\begin{equation} \label{iep+}
\mathcal{S}^+_{\mathrm{iep}} = \mathcal{S}_{\mathrm{iep}} \cup \{Z_i Z_j, ~ i,j=1,\dots,q, ~ i \neq j\}.
\end{equation}
The matrix equations $Z_i Z_j = 0$ are satisfied for $i \neq j$ by every solution of $\mathcal{S}_{\mathrm{iep}}$ as a consequence of the vanishing of $f_1, f_2^{(i)}, f_3^{(i)}$.  However, despite being of low degree, the matrix polynomials $Z_i Z_j, ~ i \neq j$ are not contained in $\mathcal{I}_1, \mathcal{I}_2$.  Consequently, incorporating these degree-two equations offers the prospect of strengthening our relaxations without a significant additional computational expense.  We define truncated ideals $\mathcal{I}_1^+, \mathcal{I}_2^+$ corresponding to $\mathcal{S}^+_{\mathrm{iep}}$ in an identical fashion to $\mathcal{I}_1, \mathcal{I}_2$ by restricting the coefficients corresponding to the additional polynomials $Z_i Z_j, ~ i \neq j$ to be matrices of constant polynomials (as in the restriction of the coefficients $h_3^{(i)}$ of $f_3^{(i)}$), with the coefficients of the other polynomials $f_1, f_2^{(i)}, f_3^{(i)}, f_4^{(k)}$ being as in $\mathcal{I}_1, \mathcal{I}_2$.

The semidefinite relaxation $\mathcal{R}^+_1(\Lambda,\eig)$ obtained as a strong alternative to the system $-1 \in \mathcal{I}^+_1 + \Sigma$ is identical to $\mathcal{R}_1(\Lambda,\eig)$, i.e., the additional redundant constraints do not strengthen the relaxation.  However, the strong alternative to the system $-1 \in \mathcal{I}^+_2 + \Sigma$ leads to a convex outer approximation $\mathcal{R}^+_2(\Lambda,\eig)$ of $\mathcal{V}_\R(\mathcal{S}_{\mathrm{iep}})$ that is in general tighter than $\mathcal{R}_2(\Lambda,\eig)$; in addition to all the constraints that define $\mathcal{R}_2(\Lambda,\eig)$ in \eqref{EqSecondDegree}, the set $\mathcal{R}^+_2(\Lambda,\eig)$ consists of the additional constraints $\sum_{r=1}^n \trace{f_{s,r} \mathcal{W}_{i,j}(f_{t,r})} = 0, ~ i,j = 1,\dots,q, ~ i \neq j, ~ s,t=1,\dots,n$ on the variables $\mathcal{W}_{i,j}$.  Thus, a notable feature of the relaxation $\mathcal{R}^+_2(\Lambda,\eig)$ is that it is of the same size as $\mathcal{R}_2(\Lambda,\eig)$, despite providing a tighter convex outer approximation in general to $\mathcal{V}_\R(\mathcal{S}_{\mathrm{iep}})$.  We demonstrate the merits of this relaxation in the numerical experiments in this section.


\subsection{Experiments with Random Affine IEPs}\label{subsec:random}

We present an experiment on random problems instances in this subsection.
Specifically, we compare the relative power of the two relaxations described in Section~\ref{SectionRelaxations} in certifying infeasibility, or from a dual viewpoint, in approximating $\mathrm{conv}(\mathcal{V}_\R(\mathcal{S}_{\mathrm{iep}}))$.  To provide a visual illustration, we consider affine IEPs involving matrices in $\Sym^3$, with a desired spectrum of $\{-1,0,1\}$.  We begin by considering an affine space defined by $\ell=3$ random linear equations, i.e., $\mathcal{E} = \{X \in \Sym^3 ~|~ \trace{C_k X} = 0,~ C_k\in \Sym^3, ~ k=1,\dots,\ell \}$, where the $C_k$'s have random entries.  Given the spectrum (which fixes the trace) and the affine space $\mathcal{E}$, the solution set $\mathcal{V}_\R(\mathcal{S}_{\mathrm{iep}})$ is constrained to lie in an affine space of dimension at most two in $\Sym^3$.  We then set the entries $X_{11},X_{22}$ to fixed values in the range $[-1,1]$, and check whether there exists a matrix in $\Sym^3$ with these values for $X_{11},X_{22}$ that can be expressed as $\sum_i \lambda_i Z_i$ for $(Z_1,Z_2,Z_3) \in \mathcal{R}_1(\Lambda,\eig)$ and for $(Z_1,Z_2,Z_3) \in \mathcal{R}_2(\Lambda,\eig)$. Figures~\ref{RandAff1} and \ref{RandAff2} represent two different problem instances obtained by generating two affine spaces $\mathcal{E}$ as described above, and they illustrate graphically when the relaxations succeed or fail at certifying infeasibility.  Evidently, the relaxation $\mathcal{R}_2(\Lambda,\eig)$ is successful in certifying infeasibility over a larger range of values of $X_{11},X_{22}$ than the relaxation $\mathcal{R}_1(\Lambda,\eig)$, thus illustrating its increased power (at a greater computational expense).  From a dual perspective, we have in both cases that $\mathcal{R}_1(\Lambda,\eig) \supsetneq \mathcal{R}_2(\Lambda,\eig)$.  In particular, the feasibility regions corresponding to $\mathcal{R}_1(\Lambda,\eig)$ and $\mathcal{R}_2(\Lambda,\eig)$ in Figures~\ref{RandAff1} and \ref{RandAff2} represent the projections of these sets onto the $(X_{11},X_{22})$-plane of $\Sym^3$.  In each of the two settings, we maximized $1000$ random linear functionals over $\mathcal{R}_2(\Lambda,\eig)$ and in all cases obtained an element of $\mathcal{V}_\R(\mathcal{S}_{\mathrm{iep}})$.  Consequently, it appears at least based on numerical evidence that $\mathcal{R}_2(\Lambda,\eig) = \mathrm{conv}(\mathcal{V}_\R(\mathcal{S}_{\mathrm{iep}}))$ in both examples.  Figures~\ref{RandAff3} and \ref{RandAff4} give two examples based on the same setup as above, but with $\ell=2$ random linear equations defining the affine space $\mathcal{E}$.  Here the dimension of the solution set $\mathcal{V}_\R(\mathcal{S}_{\mathrm{iep}})$ is at most three, and the feasibility regions corresponding to $\mathcal{R}_1(\Lambda,\eig)$ and $\mathcal{R}_2(\Lambda,\eig)$ in Figures~\ref{RandAff3} and \ref{RandAff4} represent two-dimensional projections (onto the $(X_{11},X_{22})$-plane of $\Sym^3$) of these sets.  As with the previous examples, we maximized $1000$ random linear functionals over $\mathcal{R}_2(\Lambda,\eig)$ and in all cases obtained an element of $\mathcal{V}_\R(\mathcal{S}_{\mathrm{iep}})$.  Consequently, it again appears that $\mathcal{R}_2(\Lambda,\eig) = \mathrm{conv}(\mathcal{V}_\R(\mathcal{S}_{\mathrm{iep}}))$.

\begin{figure}[hbt]
\centering
\subcaptionbox{\label{RandAff1}}{\includegraphics[width=0.24\textwidth]{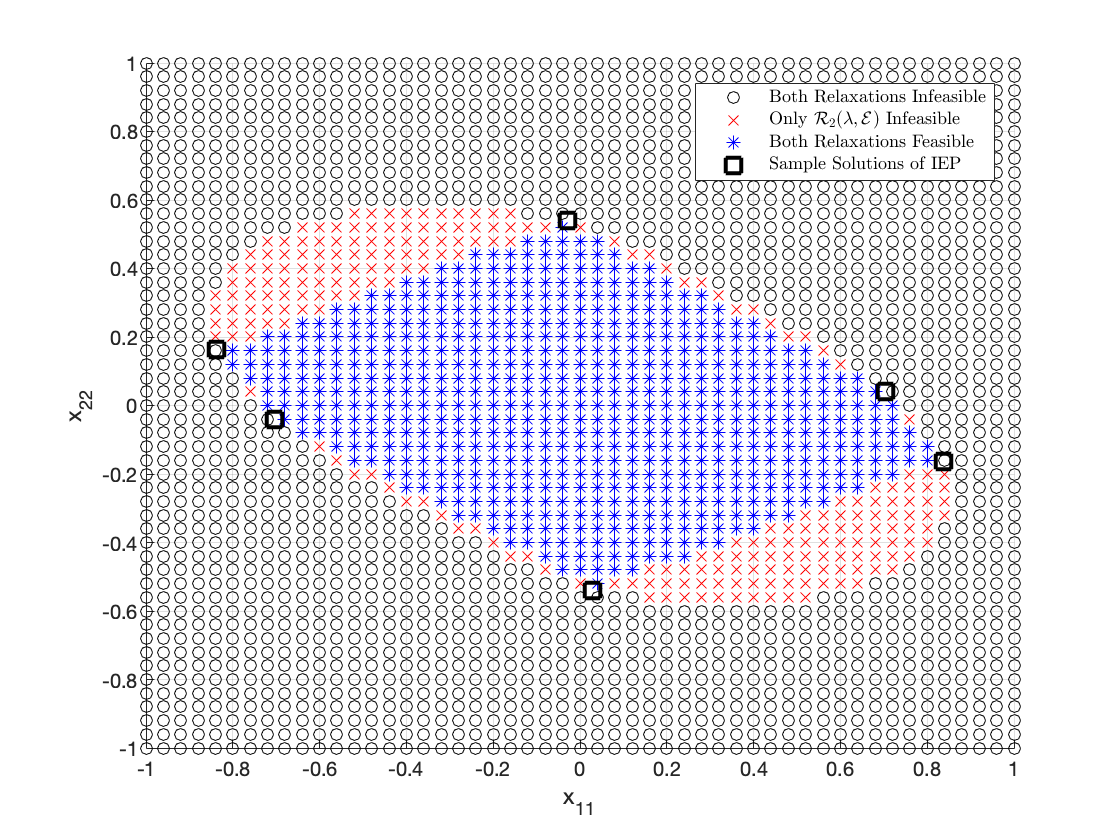}} 
\subcaptionbox{\label{RandAff2}}{\includegraphics[width=0.24\textwidth]{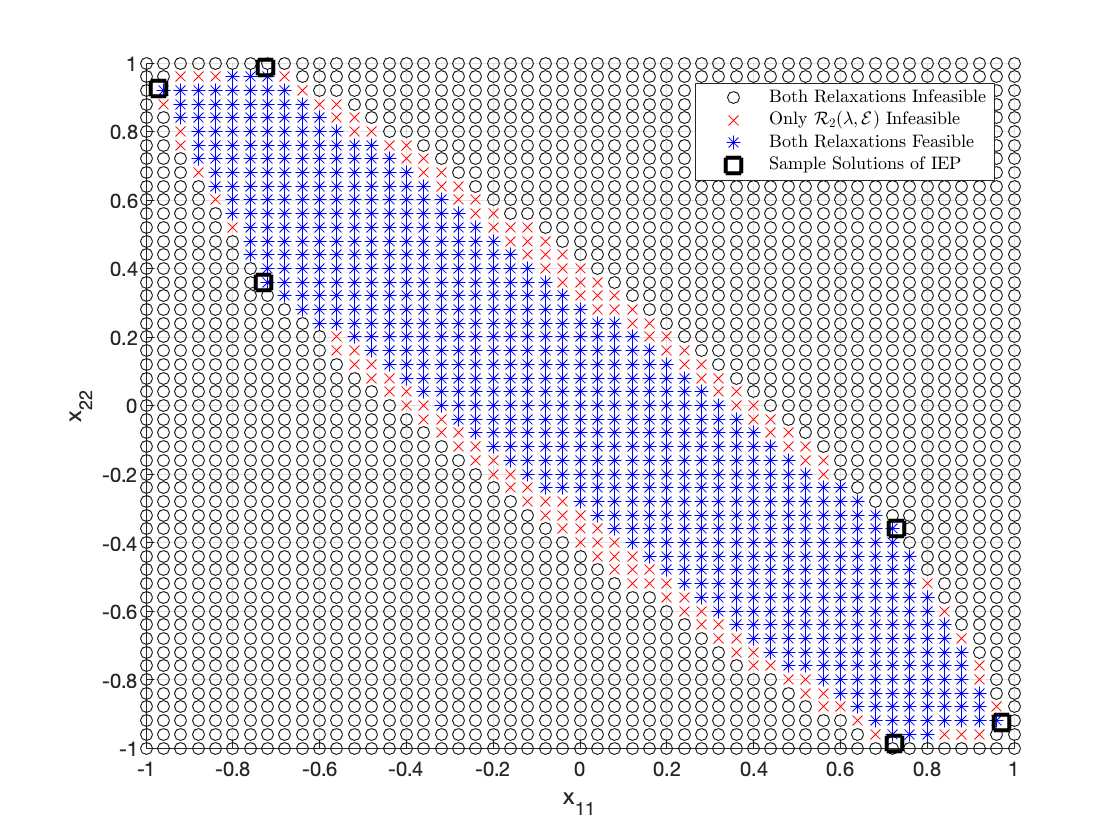}} 
\subcaptionbox{\label{RandAff3}}{\includegraphics[width=0.24\textwidth]{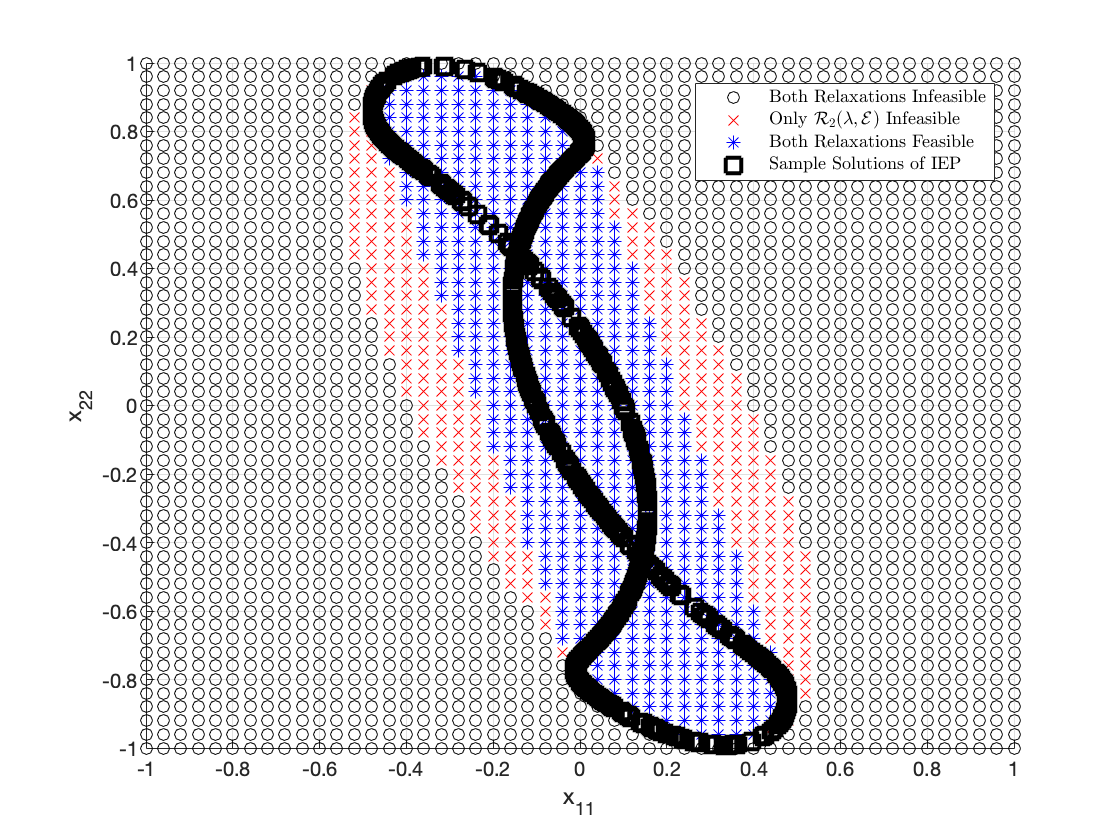}} 
\subcaptionbox{\label{RandAff4}}{\includegraphics[width=0.24\textwidth]{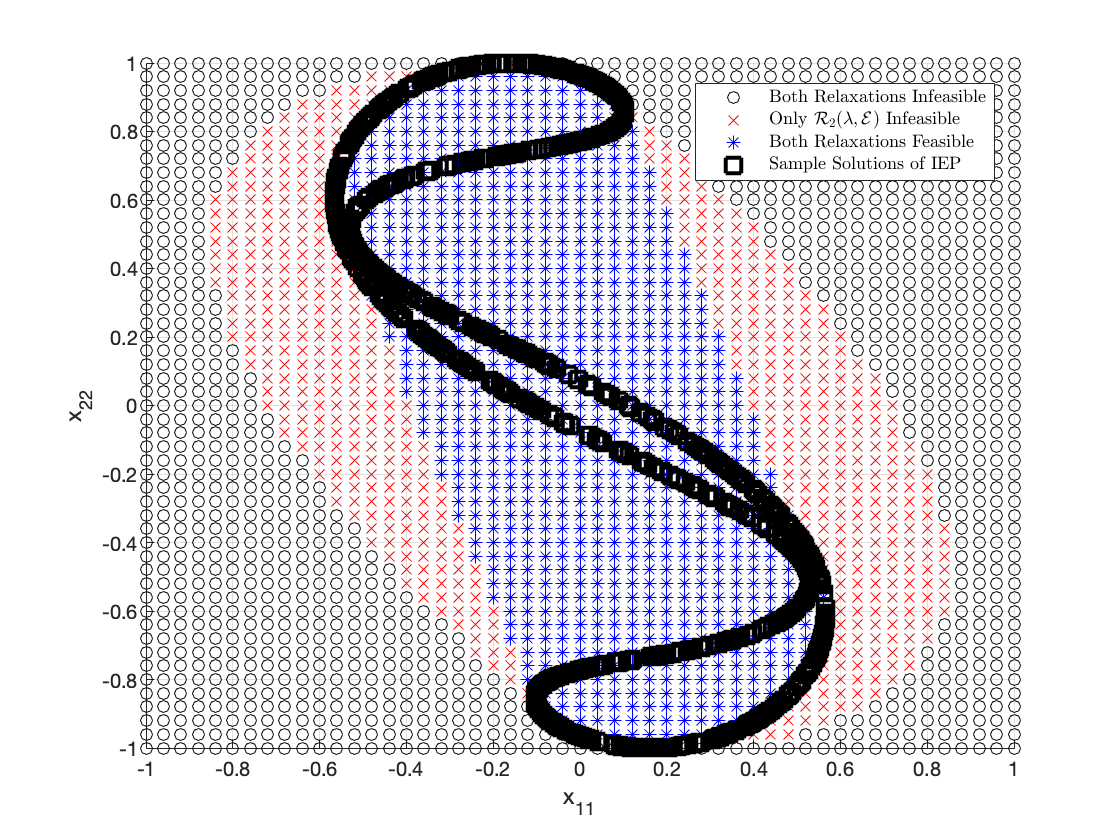}} 
\caption{Comparison of feasible/infeasible regions of $\mathcal{R}_1(\Lambda,\eig)$ and $\mathcal{R}_2(\Lambda,\eig)$ for four random problem instances as described in Section~\ref{subsec:random}. The points marked with black circles, red crosses, and blue stars correspond, respectively, to settings in which $\mathcal{R}_1(\Lambda,\eig)$ and $\mathcal{R}_2(\Lambda,\eig)$ are both infeasible; $\mathcal{R}_1(\Lambda,\eig)$ is feasible and $\mathcal{R}_2(\Lambda,\eig)$ is infeasible; and both $\mathcal{R}_1(\Lambda,\eig)$ and $\mathcal{R}_2(\Lambda,\eig)$ are feasible. Thick black squares represent $(X_{11},X_{22})$ values of solutions to the affine IEP. }
\label{FigureExpRandAff1}
\end{figure}

\subsection{Discrete Inverse Sturm-Liouville Problem}

Next, we demonstrate an application of our framework to certify infeasibility of, or produce a solution to, the extensively studied discrete inverse Sturm-Liouville problem \cite{hald1977discrete}.  This problem arises as a discretization of a continuous differential boundary problem of the form $-u''(x) + p(x)u(x) = \lambda u(x), ~ u(0) = u(\pi) =0$.  Here, $u(x)$ and $p(x)$ are functions, and $\lambda$ is a parameter that is an eigenvalue of the system. A particular discretization of this differential equation gives rise to the linear system $\left(\frac{(n+1)^2}{\pi^2} J + D \right)u = \lambda u$, where $J$ is a Jacobian matrix with diagonal entries equal to $2$ and the nonzero off-diagonal entries equal to $-1$ \cite{hald1977discrete}.  Hence, given a collection $\lambda_1, \dots, \lambda_n \in \mathbb{R}$, one wishes to identify a diagonal matrix $D$ so that this linear system has a solution for each setting $\lambda = \lambda_i$, i.e., $\lambda_1,\dots,\lambda_n$ are eigenvalues of the matrix $\frac{(n+1)^2}{\pi^2} J + D$.  This is clearly an instance of an affine IEP.

We consider two different instantiations of the problem with $n=5$.  First, we consider the set of eigenvalues $\{1,2,3,4,5\}$.  In this instance, there exists a decomposition $-1 \in \mathcal{I}_1 + \Sigma$ which certifies that the discrete inverse Sturm-Liouville problem is infeasible with the given spectrum.  Next, we consider eigenvalues in the set $\{1,4,9,16,25\}$.  In this case, the discrete inverse Sturm-Liouville problem turns out to be feasible.  Specifically, we attempt to produce a solution to the inverse discrete Sturm-Lioville problem by maximizing $100$ random linear functionals over the convex sets $\mathcal{R}_1(\Lambda,\eig)$, $\mathcal{R}_2(\Lambda,\eig)$, and $\mathcal{R}_2^+(\Lambda,\eig)$; our approach succeeds $14$ out of $100$ times over $\mathcal{R}_1(\Lambda,\eig)$, $26$ out of $100$ times over $\mathcal{R}_2(\Lambda,\eig)$, and $55$ out of $100$ times over $\mathcal{R}_2^+(\Lambda,\eig)$.  These results suggest that our semidefinite relaxations may offer a useful solution framework across the range of applications in which the discrete inverse Sturm-Liouville problem arises.

\subsection{Induced Subgraph Isomorphism}

We present next the utility of our framework in the context of a problem in combinatorial optimization, namely the induced subgraph isomorphism problem.  Here we are given two undirected, unweighted graphs $\mathcal{G}$ and $\mathcal{G}'$ on $n$ and $n'$ vertices, respectively, with $n' < n$.  The problem is to determine whether $\mathcal{G}'$ is an induced subgraph of $\mathcal{G}$.  This problem is NP-complete in general and has received considerable attention.

Suppose $\mathcal{G}'$ is an induced subgraph of $\mathcal{G}$.  Letting $A \in \Sym^{n}$ and $A' \in \Sym^{n'}$ be adjacency matrices representing the graphs $\mathcal{G}$ and $\mathcal{G}'$, respectively, such that $A'$ is equal to a principal submatrix of $A$, there must exist a matrix $M \in \Sym^n$ that satisfies the following conditions:
\begin{equation} \label{EqInducedSubgraph}
\begin{aligned}
\trace{A M} = \sum_{i,j = 1}^{n'} \left(A'\right)_{i,j}; ~~~ \left(M\right)_{i,j} = 0 \text{ if } \left(A \right)_{i,j} = 0, ~ i,j = 1,\dots,n.
\end{aligned}
\end{equation}
This consequence follows because we may choose $M$ to be equal to $A'$ on the $n' \times n'$ principal submatrix corresponding to corresponding to $\mathcal{G}'$ and zero elsewhere.  Thus, a sufficient condition to certify that $\mathcal{G}'$ is not an induced subgraph of $\mathcal{G}$ is to certifying the infeasibility of an affine IEP in which the spectrum is equal to that of $A'$ along with an eigenvalue of zero with multiplicity $n - n'$ and the affine space is given by \eqref{EqInducedSubgraph}.

With this approach, we prove that the octahedral graph with $6$ nodes and $12$ edges (shown in Figure \ref{FigureOctahedral}) is not contained as an induced subgraph in either of the larger graphs shown in Figure~\ref{FigureDeg1Suc} (on $20$ nodes with $44$ edges) and Figure~\ref{FigureDeg2Suc} (on $15$ nodes with $38$ edges).  Both of these larger graphs are randomly generated Erd\"os-Renyi random graphs where any two vertices are independently and randomly connected with probability $0.2$ for Figure \ref{FigureDeg1Suc} and $0.4$ for Figure \ref{FigureDeg2Suc}. For the first graph, there exists a decomposition $-1 \in \mathcal{I}_1 + \Sigma$, thus certifying that the octahedral graph is not an induced subgraph.  For the second graph, there is no infeasibility certificate of the form $-1 \in \mathcal{I}_1 + \Sigma$ but there is one of the form $-1 \in \mathcal{I}_2^+ + \Sigma$, thus providing a certificate that the octahedral graph is again not an induced subgraph.

\begin{figure}[hbt]
\centering
\subcaptionbox{\label{FigureOctahedral}}{\includegraphics[width=0.16\textwidth]{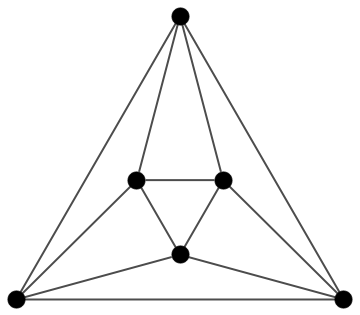}}  \hspace{0.35in}
\subcaptionbox{\label{FigureDeg1Suc}} {\includegraphics[width=0.2\textwidth]{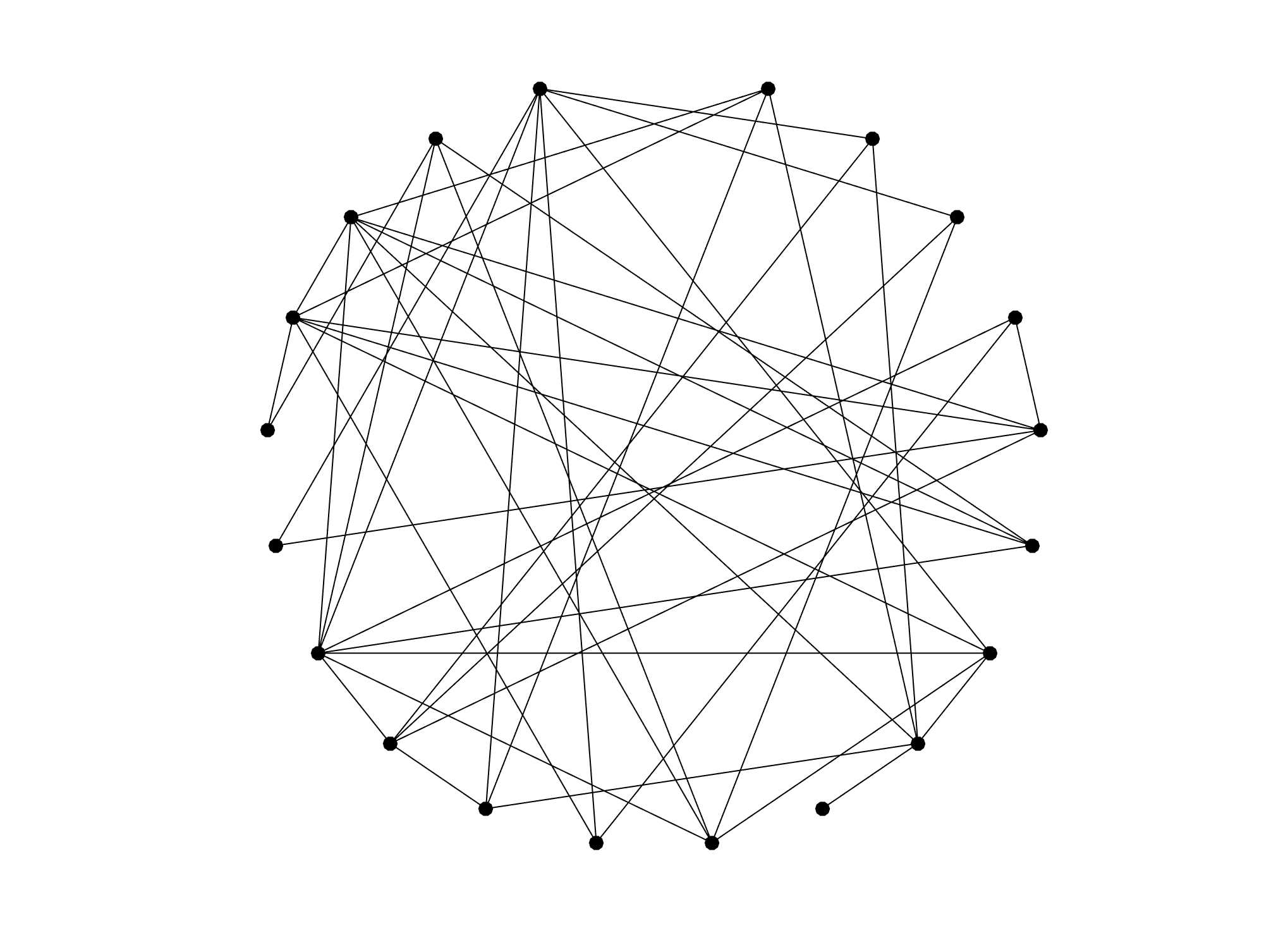}}\hspace{0.3in}
\subcaptionbox{\label{FigureDeg2Suc}} {\includegraphics[width=0.2\textwidth]{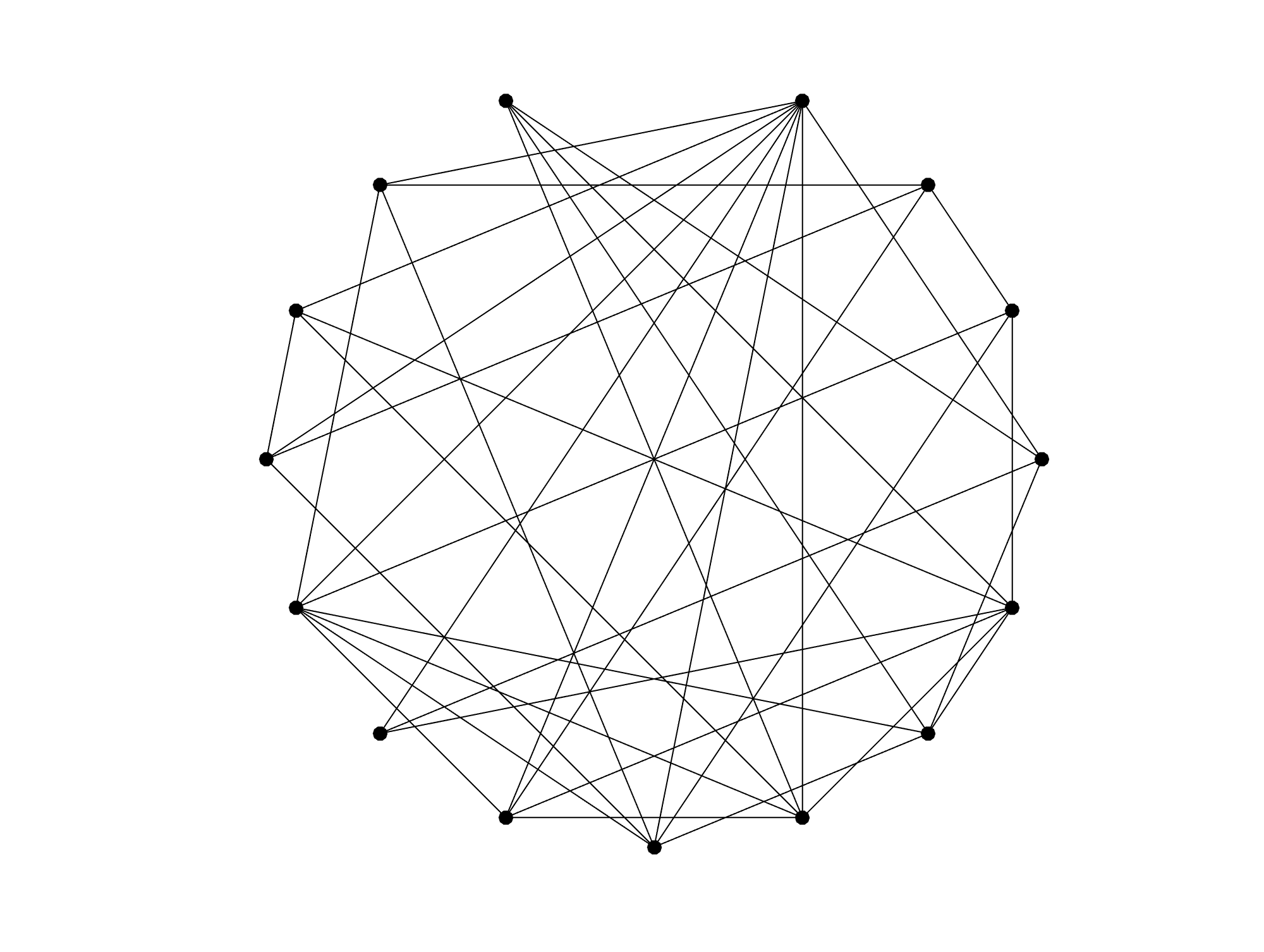}}
\caption{From left to right:  the octahedral graph, an Erd\"os-Renyi random graph on 20 nodes with p = 0.2, an Erd\"os-Renyi random graph on 15 nodes with p = 0.4.  Our first convex relaxation certifies that the octahedral graph is not an induced subgraph of the graph shown in Figure \ref{FigureDeg1Suc}. A tighter convex relaxation proves the same result for the graph shown in Figure \ref{FigureDeg2Suc}.}\label{FigureInducedSG}
\end{figure}

\subsection{Constructing a Real Symmetric Toeplitz Matrix with Desired Spectrum}
Finally, we describe how our framework can be utilized for constructing real symmetric Toeplitz matrices with a desired spectrum.  As Toeplitz matrices form a subspace, this question is an instance of an affine IEP.  Landau showed that there exists a Toeplitz matrix with a desired spectrum, but his proof was non-constructive \cite{landau1994inverse}, and numerically constructing such matrices continues to remain a challenge.

In our first experiment, we set $n=5$ and consider the problem of constructing a symmetric Toeplitz matrix with eigenvalues $\{1,2,3,4,5\}$.  We maximize random linear functionals over the sets $\mathcal{R}_{1}(\Lambda,\eig)$, $\mathcal{R}_{2}(\Lambda,\eig)$ and $\mathcal{R}^+_{2}(\Lambda,\eig)$, and we succeed at identifying a Toeplitz matrix with the desired spectrum $4$ out of $100$ times with $\mathcal{R}_{1}(\Lambda,\eig)$, $12$ out of $100$ times with $\mathcal{R}_{2}(\Lambda,\eig)$, and $41$ out of $100$ times with $\mathcal{R}_{2}^+(\Lambda,\eig)$.  In our second experiment we set $n=8$ and we seek a Toeplitz matrix with eigenvalues $-1$ (with multiplicity four) and $1$ (with multiplicity four).  With the same approach as before of maximizing random linear functionals, we identify a Toeplitz matrix with the desired spectrum $17$ out of $100$ times with $\mathcal{R}_{1}(\Lambda,\eig)$, and $84$ out of $100$ times with both $\mathcal{R}_{2}(\Lambda,\eig)$ and $\mathcal{R}_{2}^+(\Lambda,\eig)$.  In summary, our framework provides a numerical counterpart to Landau's non-constructive existence result.


\section{Conclusions} \label{sec:conc}
In this short note we describe a new framework for the affine IEP by first formulating it as a system of polynomial equations and then employing techniques from the polynomial optimization literature to obtain several semidefinite relaxations.  These relaxations offer increasingly tighter approximations at the expense of solving larger semidefinite programming problems.  We compare these relaxations both in random problem instances as well as in stylized examples in the context of various applications.

A number of future directions arise from our work.  First, it is of interest to identify conditions on a spectrum $\Lambda$ and an affine space $\eig$ so that a particular relaxation such as $\mathcal{R}_1(\Lambda,\eig)$ is tight, i.e., $\mathcal{R}_1(\Lambda,\eig) = \mathrm{conv}(\mathcal{V}_\R(\mathcal{S}_{\mathrm{iep}}))$.  These would correspond to families of instances of the affine IEP that are exactly solved by a tractable semidefinite program.  A related second question is that in considering a sequence of truncated ideals that are subsets of $\langle \mathcal{S}_{\mathrm{iep}} \rangle$, does there exist a truncated ideal $\mathcal{I}$ of bounded (but possibly large) degree coefficients for which the alternative of the system $-1 \in \mathcal{I} + \Sigma$ is exactly equal to $\mathrm{conv}(\mathcal{V}_\R(\mathcal{S}_{\mathrm{iep}}))$?  Such a property is sometimes called \emph{finite convergence} in the polynomial optimization literature; it has been shown to be true if $\mathcal{V}_\R(\mathcal{S}_{\mathrm{iep}})$ is finite, but more generally, depends on the particular structure of the problem at hand.  If this finite convergence property is true for the affine IEP setting considered in this note, then our heuristic for obtaining a solution to the system $\mathcal{S}_{\mathrm{iep}}$ (if it is feasible) based on maximizing random linear functionals over convex outer approximations of $\mathcal{V}_\R(\mathcal{S}_{\mathrm{iep}})$ is guaranteed to succeed after finitely many steps.



\bibliographystyle{plain}
\bibliography{IEPbib}

\end{document}